\documentclass[11pt]{amsart}

\usepackage{latexsym, amssymb, amsmath, amscd, amsfonts}
\usepackage{cite}
\usepackage{graphicx}
\usepackage{youngtab}
\usepackage{tabmac}

\input xy
\xyoption{all}

\usepackage{tikz}
\usepackage{color}
\definecolor{light-gray}{gray}{0.70}
\definecolor{dark-gray}{gray}{0.38}

\numberwithin{equation}{section}
\newtheorem{theorem}{Theorem}[section]
\newtheorem{corollary}[theorem]{Corollary}
\newtheorem{definition}[theorem]{Definition}
\newtheorem{example}[theorem]{Example}
\newtheorem{lemma}[theorem]{Lemma}
\newtheorem{proposition}[theorem]{Proposition}

\begin{document}

\title{Hibi Algebras and Representation Theory}

\author{Sangjib Kim}
\email{sk23@korea.ac.kr}
\address{Department of Mathematics, Korea University, Seoul 02841, South Korea}

\author{Victor Protsak}
\email{vp35@cornell.edu}
\address{Department of Mathematics, Cornell University, Ithaca, NY 14853, USA}

\begin{abstract}
This paper gives a survey on the relation between Hibi algebras and representation theory. 
The notion of Hodge algebras or algebras with straightening laws has been proved 
to be very useful to describe the structure of many important algebras 
in classical invariant theory and representation theory \cite{BH93, Ei80, DEP82, GL01, Se07}. 
In particular, a special type of such algebras introduced 
by Hibi \cite{Hi87} provides a nice bridge between combinatorics and 
representation theory of classical groups.
We will examine certain poset structures of Young tableaux and affine monoids,  
Hibi algebras in toric degenerations of flag varieties, and 
their relations to polynomial representations of the complex general linear group.
\end{abstract}


\maketitle

\section{Young tableaux and Gelfand-Tsetlin poset}

In this section, we will define some partially ordered sets and 
investigate their properties.

\subsection{Poset of column tableaux} \label{sec-yd-tab}

A \emph{Young diagram} $\lambda$ is a collection of boxes arranged 
in left-justified rows with the row lengths in non-increasing order. 
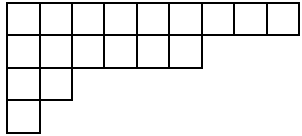
\[
\young(\ \ \ \ \ \ \ \ \ ,\ \ \ \ \ \ ,\ \ ,\ )
\]
Writing $\lambda_i$ for the length of the $i$th row of $\lambda$ 
counting from top to bottom, we will identify $\lambda$ 
with a non-increasing sequence of integers 
\[
\lambda=(\lambda_1, \lambda_2, ...) \quad
     \text{such that $\lambda_1 \geq \lambda_2 \geq \cdots \geq 0$}.
\]
The \emph{transpose} (or \emph{conjugate}) $\lambda^t$ of a Young diagram $\lambda$ is 
the Young diagram $(d_1, d_2, ...)$ where $d_j$ is the number of boxes 
in the $j$th column of $\lambda$ counting from left to right.
The \emph{depth} of $\lambda$ is the number of non-empty rows 
in $\lambda$ and will be denoted by $\mathbf{d}(\lambda)$.

A \emph{Young tableaux} is a filling of the boxes of a Young diagram 
with positive integers. The \emph{content} of a Young tableau $T$ is 
a sequence $(c_1, c_2, ...)$ where $c_i$ is the number of boxes containing $i$ in $T$. 
A Young tableau is called \emph{semistandard} if its entries in each row  weakly 
increase from left to right and its entries in each column strictly increase 
from top to bottom. 

From now on, we fix a positive integer $n$ and then consider Young diagrams whose 
depths are not more than $n$ and Young tableaux whose entries are from $\{1, 2, ..., n\}$.
For example,  when $n=6$, the following is a semistandard Young tableau 
\begin{equation*}
\young(111122333566,2223335566,333555,5566)
\end{equation*}
on a  Young diagram $(12,10,6,4,0,0)$ with content $(4,5,9,0,8,6)$.

For notational convenience, for $1\leq k \leq n$,  
we let $(1^k)$ or $(k)^t$ denote a Young diagram having $k$ boxes in a single column, 
and write $[i_1, i_2, ..., i_k]$ for a semistandard 
tableau on a Young diagram $(1^k)$ whose $j$th entry counting 
from top to bottom is $i_j$ for $1\leq j \leq k$. 
We will call $[i_1, i_2,..., i_k]$ a \emph{column tableau of depth} $k$. 
See Figure \ref{fig-column-tab}.

\begin{figure}
\begin{equation*}
(1^k) = (k)^t = \tableau{\ \\  \  \\ \vdots \\  \ } \qquad \ \text{and} \qquad
[i_1, i_2, ..., i_k] = \tableau{i_1\\  i_2 \\ \vdots \\  i_k}
\end{equation*}
\caption{A Young diagram with a single column having $k$ boxes and  
              a column tableau of depth $k$.}
\label{fig-column-tab}
\end{figure}

\begin{definition}
The \emph{poset of column tableaux} is the set of column tableaux
\[
L_n = \bigcup_{1 \leq k \leq n} 
      \left\{ [i_1, i_2, ..., i_k] : 1 \leq i_1 < i_2 < \cdots < i_k \leq n\right\}
\]
with the following partial order. For $I, J \in L_n$ where
\[
I=[i_1, i_2, ..., i_a] \quad \text{and}  \quad J=[j_1,j_2,...,j_b],
\] 
we let $I \geq_{tab} J$ if $a \leq b$ and $i_{\ell} \geq j_{\ell}$ for all $1\leq \ell \leq a$. 
\end{definition}
It is straightforward to check that the poset $(L_n, \geq_{tab})$ 
forms a distributive lattice \cite{GL01}. For $I=[i_1, i_2, ..., i_a]$ and $J=[j_1,j_2,...,j_b]$ 
in $L_n$ with $a \leq b$, their join and meet are
\[
I \vee J = [x_1, ..., x_a] \quad \text{and} \quad I \wedge J =[y_1, ..., y_b]
\]
respectively where 
\[
 x_{\ell} = \max(i_{\ell}, j_{\ell}) \quad {\text{for\ }} 1\leq \ell \leq a,
\qquad
 y_{\ell}  =
 \begin{cases}
  \min(i_{\ell}, j_{\ell})  & {\text{for\ }} 1\leq \ell \leq a, \\ 
  \ \ j_{\ell}              & {\text{for\ }} a < \ell \leq b. 
  \end{cases}
\]
We will call $L_n$ a \emph{distributive lattice of column tableaux}. 
See Figure \ref{fig:Hasse-L4}.

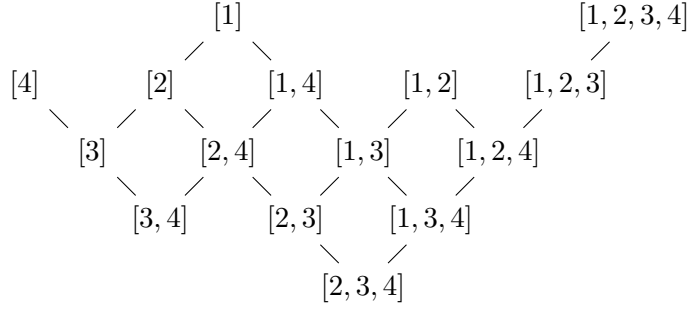
\begin{figure}[!ht]
\centering
\begin{tikzpicture}[scale=.9]
 \node (4) at (-4,1) {$[4]$};
  \node (3) at (-3,0) {$[3]$};  
  \node (34) at (-2,-1) {$[3,4]$};  
  \node (2) at (-2,1) {$[2]$};
  \node (24) at (-1,0) {$[2,4]$};  
  \node (1) at (-1,2) {$[1]$};  
  \node (23) at (0,-1) {$[2,3]$};
  \node (14) at (0,1) {$[1,4]$};  
  \node (234) at (1,-2) {$[2,3,4]$};  
  \node (13) at (1,0) {$[1,3]$};
  \node (134) at (2,-1) {$[1,3,4]$};  
  \node (12) at (2,1) {$[1,2]$};  
  \node (124) at (3,0) {$[1,2,4]$};
  \node (123) at (4,1) {$[1,2,3]$};  
  \node (1234) at (5,2) {$[1,2,3,4]$};  
 \draw (4) -- (3) -- (2) -- (1) -- (14) -- (13) -- (12) -- (124) -- (123) -- (1234);  
 \draw (2) -- (24) -- (14);
 \draw (3)  -- (34) -- (24) -- (23) -- (13) -- (134) -- (124);  
 \draw (23) -- (234) -- (134);
\end{tikzpicture}
\caption{The Hasse diagram of ${L}_4$. 
              The elements decrease along the lines from left to right.}
\label{fig:Hasse-L4}
\end{figure}

We remark that every multichain (i.e., linearly ordered multisubset) of 
$L_n$ can be identified with a semistandard Young tableau. 
For every multichain, by concatenating its elements in weakly increasing order,
we obtain a semistandard tableau. 
Conversely, from the definition of semistandard Young tableaux, 
the columns of every semistandard Young tableau form a multichain 
of $L_n$ with respect to the order $\geq_{tab}$. See Figure \ref{fig:chain-ssyt}.
\begin{figure}[!ht]
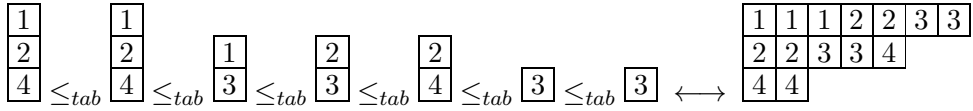

\centering
\[
\young(1,2,4) \leq_{tab} \young(1,2,4) \leq_{tab} \young(1,3) \leq_{tab} \young(2,3) 
\leq_{tab} \young(2,4) \leq_{tab} \young(3) \leq_{tab} \young(3)  
\; \longleftrightarrow \; \young(1112233,22334,44)
\]
\caption{A multichain of $L_4$ and a semistandard Young tableau.}
\label{fig:chain-ssyt}
\end{figure}

\subsection{GT poset and indicator functions} \label{sec-GT-indicator}

Usually, a Gelfand-Tsetlin (GT) pattern is defined as a triangular array of integers satisfying 
certain inequalities \cite{GT50, Mo06}. Here, we want to define it using a poset.

\begin{definition}\label{def-GT-patterns-poset}
\begin{enumerate}
\item The \emph{GT poset} $\Gamma_n$ is the set
\[
\Gamma_n = \left\{ z^{(i)}_j : 1 \leq j \leq i \leq n \right\}
\]
with the partial order $z_{j}^{(i+1)} \geq z_{j}^{(i)} \geq z_{j+1}^{(i+1)}$ 
for all $1 \leq j \leq i \leq n-1$.

\item A \emph{GT pattern} is  an order-preserving map from $\Gamma_n$ to non-negative integers, 
that is, a map $f: \Gamma_n \longrightarrow \mathbb{Z}_{\geq 0}$ such that 
\[
\text{$z^{(a)}_b \geq z^{(c)}_d$ in $\Gamma_n$ implies 
       that $f(z^{(a)}_b) \geq f(z^{(c)}_d)$.} 
\]
\item We let $S_n$ denote the set of all GT patterns.
\end{enumerate}
\end{definition}

We shall draw $\Gamma_n$ in the form of an inverted pyramid as in Figure \ref{fig:Hasse-GT4}. 
Then, for $f \in S_n$, by placing its values $f(z^{(i)}_j)$ at the positions of $z^{(i)}_j$ 
in $\Gamma_n$, we can identify $f$ with a triangular array of integers 
as GT patterns are usually defined. 
We can also identify  $f$ with an integral point 
$(f(z^{(i)}_j))_{1 \leq j \leq i \leq n}$ in $\mathbb{R}^{n(n+1)/2}$ 
and then $S_n$ can be considered an integral lattice cone  
in $\mathbb{R}^{n(n+1)/2}$  \cite{Ho-Weyl}.

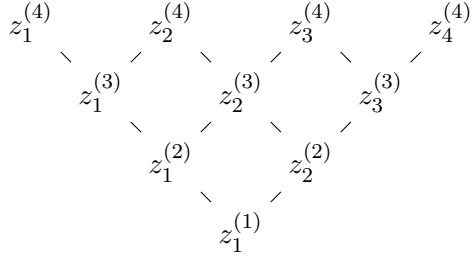
\begin{figure}[!ht]
\centering
\begin{tikzpicture}[scale=.93]
 \node (41) at (-3,2) {$z^{(4)}_1$};
 \node (42) at (-1,2) {$z^{(4)}_2$};
 \node (43) at (1,2) {$z^{(4)}_3$};
 \node (44) at (3,2) {$z^{(4)}_4$};
 \node (31) at (-2,1) {$z^{(3)}_1$};
 \node (32) at (0,1) {$z^{(3)}_2$};
 \node (33) at (2,1) {$z^{(3)}_3$};  
 \node (21) at (-1,0) {$z^{(2)}_1$};
 \node (22) at (1,0) {$z^{(2)}_2$};
 \node (11) at (0,-1) {$z^{(1)}_1$};  
   \draw (41) -- (31) -- (42) -- (32) -- (43) -- (33) -- (44);  
   \draw (31) -- (21) -- (32) -- (22) -- (33);
   \draw (21)  -- (11) -- (22);  
\end{tikzpicture}
\caption{The Hasse diagram of the GT poset $\Gamma_4$. 
              The elements decrease along the lines from left to right.}
\label{fig:Hasse-GT4}
\end{figure}

Now let us focus on GT patterns $f \in S_n$ whose images are contained in $\{ 0, 1\}$,
or equivalently,  the points in the lattice cone $S_n \subseteq \mathbb{R}^{n(n+1)/2}$ 
whose coordinates are either $1$ or $0$. Since $f$ is order-preserving, its support 
\[
Supp(f)=\{ x\in \Gamma_n: f(x) \ne 0\}
\] 
is an order-increasing subset of $\Gamma_n$, i.e., for $x,y \in \Gamma_n$, 
if $x \in Supp(f)$ and $y \geq x$, then $y \in Supp(f)$. 
In fact, for every order-increasing subset $A$ of $\Gamma_n$, its \emph{indicator function}  
$\mathbf{1} _{A} : \Gamma_n \to \{ 0,1\}$  belongs to $S_n$ where
\[
\mathbf {1} _{A}(x)=
        {\begin{cases} 
		        1&{\text{if }} x \in A,\\
		        0&{\text{if }} x \notin A.
	\end{cases}}
\]

\begin{definition} 
The \emph{poset of indicator functions} is the set
\[
\Lambda_n = \left\{ \mathbf {1} _{A} \in S_n : \text{$A$ is 
         a non-empty order-increasing subset of $\Gamma_n$} \right\}
\]
with the reverse inclusion order on the order-increasing subsets of $\Gamma_n$, that is,
\[
\mathbf {1} _{A} \geq_{ind} \mathbf {1} _{B} \text{ \ if and only if \ } A \subseteq B.
\] 
\end{definition}
Then, with the following join and meet
\begin{equation}\label{join-meet-indicator}
\mathbf {1} _{A} \vee \mathbf {1} _{B} = \mathbf {1} _{A \cap B} \quad \text{and} \quad 
\mathbf {1} _{A} \wedge \mathbf {1} _{B} = \mathbf {1} _{A \cup B}
\end{equation}
respectively, the poset $(\Lambda_n, \geq_{ind})$ is a distributive lattice. 

\begin{theorem} \label{thm:isom-birkoff}
The poset $(L_n, \geq_{tab})$ of column tableaux is order-isomorphic to 
the poset $(\Lambda_n, \geq_{ind})$ of indicator functions.
\end{theorem}
\begin{proof}
Let $I=[i_1, i_2, ..., i_k] \in L_n$. For each $1\leq a \leq n$, 
we let $\ell_a$ be the number of the entries $i_j$ in $I$ 
which are not more than $a$. Then, we consider an order-preserving map
$f_I:\Gamma_n \longrightarrow \{ 0,1\}$ such that 
for each $a$, the number of $z_{b}^{(a)}$ for $1\leq b \leq a$ such 
that  $f_I(z_{b}^{(a)})=1$ is $\ell_a$. Since we know that 
\[
f_I (t^{(a)}_1) \geq f_I (t^{(a)}_2) \geq \cdots 
                    \geq f_I (t^{(a)}_{a-1}) \geq f_I (t^{(a)}_a)
\] 
for $1\leq a \leq n$, the numbers $\ell_a$ can completely determine $f_I$. 
Then, it is straightforward to check that the map sending $I$ to $f_I$ 
gives an order-isomorphism from $L_n$ to $\Lambda_n$.  See \cite[\S 3.3]{Ki08}. 
\end{proof}

For $I \in L_n$, if we write $A$ for the support of the corresponding 
map $f_I \in \Lambda_n$, then $f_I$ is the indicator function of 
the order-increasing subset $A$ of $\Gamma_n$
\[
f_I = \mathbf{1}_{A} \quad \text{ where $A = Supp(f_I)$}.
\]
With the dual relation between order-decreasing subsets (also called order ideals) 
and order-increasing subsets, Theorem \ref{thm:isom-birkoff} is basically 
Birkhoff's representation theorem for distributive lattices (also known as
the fundamental theorem for finite distributive lattices \cite[\S 3.4]{St12}) 
applied to the distributive lattice $L_n$. 
Recall that an element in a lattice is  \emph{join-irreducible}, 
if it is neither the least  element of the lattice nor the join of 
any two smaller elements.
Then, Theorem \ref{thm:isom-birkoff} tells us that the GT poset 
$\Gamma_n$ can be identified with the set $J(L_n)$ of join-irreducible 
elements of $L_n$ with an additional greatest element $z_1^{(n)}$. 
See Figure \ref{fig:join-irred-L4} and 
compare it with Figure \ref{fig:Hasse-GT4}.

\begin{figure}
\centering
\begin{tikzpicture}[scale=.85]
  \node (top) at (-3,2) {$\bullet$};
    \node (1) at (-1,2) {$[1]$};
  \node (12) at (1,2) {$[1,2]$};
  \node (123) at (3,2) {$[1,2,3]$};
  \node (4) at (-2,1) {$[4]$};
  \node (14) at (0,1) {$[1,4]$};
  \node (124) at (2,1) {$[1,2,4]$};  
  \node (34) at (-1,0) {$[3,4]$};
  \node (134) at (1,0) {$[1,3,4]$};
  \node (234) at (0,-1) {$[2,3,4]$};  
  \draw (top) -- (4) -- (1) -- (14) -- (12) -- (124) -- (123);  
 \draw (4) -- (34) -- (14) -- (134) -- (124);
 \draw (34)  -- (234) -- (134);  
\end{tikzpicture}
\caption{The poset consisting of all the join-irreducible 
		elements of $L_4$ and an additional greatest element $\bullet$. 
        The elements decrease along the lines from left to right.}
\label{fig:join-irred-L4}
\end{figure}
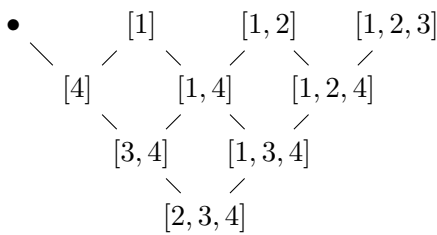

With Birkhoff's theorem,
the greatest column tableau $[n]$ in $L_n$ corresponds to the largest order 
ideal of $J(L_n)$, which is itself. Then, by taking its complement in $J(L_n)$, 
the order-increasing subset of $J(L_n)$ corresponding to $[n]$ is the empty set.
For us,  the column tableau $[n]$ corresponds to the indicator function 
of the singleton set $\{ z_1^{(n)} \} = \Gamma_n \setminus J(L_n)$.
By considering the GT poset $\Gamma_n$ rather than $J(L_n)$, we can reserve 
the indicator function of the empty set defined on $\Gamma_n$ 
for the identity in the monoid $S_n$ we will study in the following section.

\begin{example} \label{ex:col-tab-GT3}
The following gives an order isomorphism between $L_3$ and $\Lambda_3$.
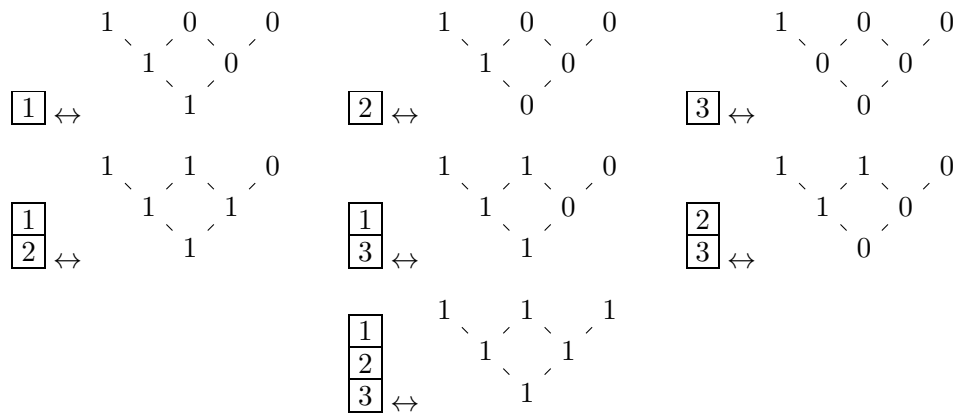
\begin{align*}
\young(1) \leftrightarrow
\begin{tikzpicture}[scale=.55]
\node (31) at (-2,1) {$1$};
\node (32) at (0,1)  {$0$};
\node (33) at (2,1)  {$0$};  
\node (21) at (-1,0) {$1$};
\node (22) at (1,0)  {$0$};
\node (11) at (0,-1) {$1$};  
\draw (31) -- (21) -- (32) -- (22) -- (33);
\draw (21)  -- (11) -- (22);  
\end{tikzpicture}
\qquad
&\young(2) \leftrightarrow
\begin{tikzpicture}[scale=.55]
\node (31) at (-2,1) {$1$};
\node (32) at (0,1)  {$0$};
\node (33) at (2,1)  {$0$};  
\node (21) at (-1,0) {$1$};
\node (22) at (1,0)  {$0$};
\node (11) at (0,-1) {$0$};  
\draw (31) -- (21) -- (32) -- (22) -- (33);
\draw (21)  -- (11) -- (22);  
\end{tikzpicture}
\qquad
\young(3) \leftrightarrow
\begin{tikzpicture}[scale=.55]
\node (31) at (-2,1) {$1$};
\node (32) at (0,1)  {$0$};
\node (33) at (2,1)  {$0$};  
\node (21) at (-1,0) {$0$};
\node (22) at (1,0)  {$0$};
\node (11) at (0,-1) {$0$};  
\draw (31) -- (21) -- (32) -- (22) -- (33);
\draw (21)  -- (11) -- (22);  
\end{tikzpicture} \\
\young(1,2) \leftrightarrow
\begin{tikzpicture}[scale=.55]
\node (31) at (-2,1) {$1$};
\node (32) at (0,1)  {$1$};
\node (33) at (2,1)  {$0$};  
\node (21) at (-1,0) {$1$};
\node (22) at (1,0)  {$1$};
\node (11) at (0,-1) {$1$};  
\draw (31) -- (21) -- (32) -- (22) -- (33);
\draw (21)  -- (11) -- (22);  
\end{tikzpicture}
\qquad
& \young(1,3) \leftrightarrow
\begin{tikzpicture}[scale=.55]
\node (31) at (-2,1) {$1$};
\node (32) at (0,1)  {$1$};
\node (33) at (2,1)  {$0$};  
\node (21) at (-1,0) {$1$};
\node (22) at (1,0)  {$0$};
\node (11) at (0,-1) {$1$};  
\draw (31) -- (21) -- (32) -- (22) -- (33);
\draw (21)  -- (11) -- (22);  
\end{tikzpicture}
\qquad
\young(2,3) \leftrightarrow
\begin{tikzpicture}[scale=.55]
\node (31) at (-2,1) {$1$};
\node (32) at (0,1)  {$1$};
\node (33) at (2,1)  {$0$};  
\node (21) at (-1,0) {$1$};
\node (22) at (1,0)  {$0$};
\node (11) at (0,-1) {$0$};  
\draw (31) -- (21) -- (32) -- (22) -- (33);
\draw (21)  -- (11) -- (22);  
\end{tikzpicture}
\\
&\young(1,2,3) \leftrightarrow
\begin{tikzpicture}[scale=.55]
\node (31) at (-2,1) {$1$};
\node (32) at (0,1)  {$1$};
\node (33) at (2,1)  {$1$};  
\node (21) at (-1,0) {$1$};
\node (22) at (1,0)  {$1$};
\node (11) at (0,-1) {$1$};  
\draw (31) -- (21) -- (32) -- (22) -- (33);
\draw (21)  -- (11) -- (22);  
\end{tikzpicture}
\end{align*}
\end{example}

\bigskip

\section{Affine monoid of GT patterns and Hibi algebra}

In this section, we study the monoid structure of $S_n$, the Hibi algebra on $L_n$,
and their properties.

\subsection{Affine monoid of GT patterns}\label{Sec-SSYT-GT}

The sum of any two order-preserving maps is again order-preserving, and therefore 
the set $S_n$ of all GT patterns can be considered an affine semigroup with respect to 
the usual addition of functions. The zero map is its identity. 

\begin{definition}
The \emph{affine monoid of GT patterns} is the set of all GT patterns
\[
S_n = \left\{ f : \Gamma_n \rightarrow \mathbb{Z}_{\geq 0} \, 
                   | \, \text{$f$ is order-preserving} \right\}
\]
with the usual addition of functions.
\end{definition}

For a semistandard Young tableau $T$, let $I_j$ be its $j$th column 
counting from left to right. Then they form a multichain of $L_n$ 
as we remarked at the end of \S \ref{sec-yd-tab}
\[
I_1 \leq_{tab} I_2 \leq_{tab} \cdots \leq_{tab} I_k.
\]
If we let $\mathbf{1}_{A_j}=f_{I_j}$ be the indicator functions 
corresponding to the column tableaux $I_j$ given in Theorem \ref{thm:isom-birkoff}, 
then they form a multichain of $\Lambda_n$ and their sum
\[
\mathbf{1}_{A_1} + \mathbf{1}_{A_2} + \cdots + \mathbf{1}_{A_k} 
\quad \text{with $A_1 \supseteq A_2 \supseteq \cdots \supseteq A_k$}
\]
is again an order-preserving map. 
This is an extension of the bijection in Theorem \ref{thm:isom-birkoff},
and it gives a correspondence between the multichains of $L_n$ and 
the elements in $S_n$.

\begin{proposition}\label{prop-ssyt-gt}
There is a bijection between semistandard Young tableaux and GT patterns.
\end{proposition}
\begin{proof}
With the above discussion, it is enough to show that every GT pattern can be expressed 
as a sum of linearly ordered elements in $\Lambda_n$. 
Let  $f \in S_n$ and its image be $\{i_1, ..., i_m\}$ with $0 \leq i_1 < i_2 < \cdots < i_m$. 
If we write $A_k$ for the inverse image of $\{ y \in \mathbb{Z}: y \geq i_k \}$ under $f$, 
then they are order-increasing subsets of $\Gamma_n$ 
and satisfy $A_1 \supset A_2 \supset \cdots \supset A_m$. Therefore, their indicator functions 
$\mathbf{1}_{A_k}$  form a multichain of $\Lambda_n$. Now $f$ can be expressed as 
\begin{equation}\label{eq-sum}
f = c_1 \mathbf{1}_{A_1} + c_2 \mathbf{1}_{A_2} + \cdots + c_m \mathbf{1}_{A_m} 
\end{equation}
where $c_1 = i_1$ and  $c_k =i_k - i_{k-1}$ for $2 \leq k \leq m$. 
See Example \ref{ex-ssyt-gt}.
\end{proof}

We remark that there is a well-know bijection between semistandard Young 
tableaux and GT patterns which does not refer to these poset structures. 
For a semistandard Young tableau $T$, we define $f_T: \Gamma_n \rightarrow \mathbb{Z}$ 
by 
\begin{equation} \label{eq-usual-bij}
f_T (z_{j}^{(i)})  \, =  \,
{
\begin{split} 
& \text{\ the number of entries in the $j$th row of $T$} \\
& \text{\ which are less than or equal to $i$}
\end{split}
}
\end{equation}
for all $1 \leq j \leq i \leq n$. Then, the correspondence $T \mapsto f_T$ 
gives a bijection between semistandard Young tableaux and GT patterns. 
This bijection is the same as the one given in Proposition \ref{prop-ssyt-gt} 
in terms of mutichains in the posets $L_n$ and $\Lambda_n$. 

\begin{example}\label{ex-ssyt-gt}
Using the formula \eqref{eq-sum}, we can express
\begin{align*}
\begin{tikzpicture}[scale=.53]
 \node (41) at (-3,2) {$10$};
  \node (42) at (-1,2) {$7$};
   \node (43) at (1,2) {$3$};
    \node (44) at (3,2) {$2$};
 \node (31) at (-2,1) {$7$};
  \node (32) at (0,1) {$7$};
   \node (33) at (2,1) {$2$};  
    \node (21) at (-1,0) {$7$};
  \node (22) at (1,0) {$3$};
   \node (11) at (0,-1) {$3$};  
      \draw (41) -- (31) -- (42) -- (32) -- (43) -- (33) -- (44);  
   \draw (31) -- (21) -- (32) -- (22) -- (33);
   \draw (21)  -- (11) -- (22);  
\end{tikzpicture}
\; = \;  &
\begin{tikzpicture}[scale=.53]
 \node (41) at (-3,2) {$2$};
  \node (42) at (-1,2) {$2$};
   \node (43) at (1,2) {$2$};
    \node (44) at (3,2) {$2$};
 \node (31) at (-2,1) {$2$};
  \node (32) at (0,1) {$2$};
   \node (33) at (2,1) {$2$};  
    \node (21) at (-1,0) {$2$};
  \node (22) at (1,0) {$2$};
   \node (11) at (0,-1) {$2$};  
      \draw (41) -- (31) -- (42) -- (32) -- (43) -- (33) -- (44);  
   \draw (31) -- (21) -- (32) -- (22) -- (33);
   \draw (21)  -- (11) -- (22);  
\end{tikzpicture}
\: + \; 
\begin{tikzpicture}[scale=.53]
 \node (41) at (-3,2) {$1$};
  \node (42) at (-1,2) {$1$};
   \node (43) at (1,2) {$1$};
    \node (44) at (3,2) {$0$};
 \node (31) at (-2,1) {$1$};
  \node (32) at (0,1) {$1$};
   \node (33) at (2,1) {$0$};  
    \node (21) at (-1,0) {$1$};
  \node (22) at (1,0) {$1$};
   \node (11) at (0,-1) {$1$};  
      \draw (41) -- (31) -- (42) -- (32) -- (43) -- (33) -- (44);  
   \draw (31) -- (21) -- (32) -- (22) -- (33);
   \draw (21)  -- (11) -- (22);  
\end{tikzpicture} 
\\
& 
\; + \; 
\begin{tikzpicture}[scale=.53]
 \node (41) at (-3,2) {$4$};
  \node (42) at (-1,2) {$4$};
   \node (43) at (1,2) {$0$};
    \node (44) at (3,2) {$0$};
 \node (31) at (-2,1) {$4$};
  \node (32) at (0,1) {$4$};
   \node (33) at (2,1) {$0$};  
    \node (21) at (-1,0) {$4$};
  \node (22) at (1,0) {$0$};
   \node (11) at (0,-1) {$0$};  
      \draw (41) -- (31) -- (42) -- (32) -- (43) -- (33) -- (44);  
   \draw (31) -- (21) -- (32) -- (22) -- (33);
   \draw (21)  -- (11) -- (22);  
\end{tikzpicture}
\; + \: 
\begin{tikzpicture}[scale=.53]
 \node (41) at (-3,2) {$3$};
  \node (42) at (-1,2) {$0$};
   \node (43) at (1,2) {$0$};
    \node (44) at (3,2) {$0$};
 \node (31) at (-2,1) {$0$};
  \node (32) at (0,1) {$0$};
   \node (33) at (2,1) {$0$};  
    \node (21) at (-1,0) {$0$};
  \node (22) at (1,0) {$0$};
   \node (11) at (0,-1) {$0$};  
      \draw (41) -- (31) -- (42) -- (32) -- (43) -- (33) -- (44);  
   \draw (31) -- (21) -- (32) -- (22) -- (33);
   \draw (21)  -- (11) -- (22);  
\end{tikzpicture}
\end{align*}
Then, this GT pattern corresponds to the following multichain of $L_4$
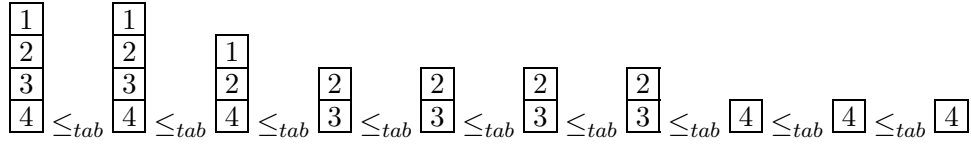
\[
\young(1,2,3,4) \leq_{tab} \young(1,2,3,4)  \leq_{tab}  \young(1,2,4)  \leq_{tab}  \young(2,3)  \leq_{tab}
 \young(2,3)  \leq_{tab}  \young(2,3) \leq_{tab}  \young(2,3)  \leq_{tab}\young(4) \leq_{tab}  \young(4)  \leq_{tab}  \young(4)
\]
or equivalently the semistandard Young tableau
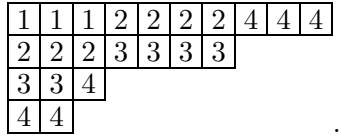
\[
\young(1112222444,2223333,334,44).
\]
\end{example}

\subsection{Hibi algebra and affine monoid algebra}
 
In \cite{Hi87}, Hibi introduced an algebra $\mathcal{H}_L$ attached to 
a finite lattice $L$, now called the \emph{Hibi algebra} on $L$. 
It is the quotient of the polynomial ring with variables $x_{\alpha}$ 
indexed by $\alpha \in L$ 
by the ideal $\mathcal{I}_L$ generated by the binomials
$x_{\alpha} x_{\beta} - x_{\alpha \wedge \beta} x_{\alpha \vee \beta}$ 
for all incomparable pairs $(\alpha, \beta)$ in $L$.

Among many others, it is shown that if $L$ is a distributive lattice 
then $\mathcal{H}_L$ is an algebra with straightening laws on $L$ and therefore 
all the monomials which are not divisible by $x_{\alpha} x_{\beta}$ 
for any incomparable pairs  $(\alpha, \beta)$ in $L$ form 
a $\mathbb{C}$-basis for $\mathcal{H}_L$.
We would like to study the Hibi algebra on the distributive lattice $L_n$ 
of column tableaux
\[
\mathcal{H}_n = \mathbb{C}[x_{I} : I \in L_n] / \mathcal{I}_{L_n}.  
\]

On the other hand, we have the affine monoid algebra $\mathbb{C}[S_n]$ of 
the monoid $S_n$ of GT patterns. 
Note that the following identities hold
\begin{equation}\label{eq-sum-j-m22}
\mathbf{1}_{A} + \mathbf{1}_{B}  
= \mathbf{1}_{A \cup B} + \mathbf{1}_{A\cap B}  
\end{equation}
for all pairs $(A,B)$ of incompariable order-increasing subsets of $\Gamma_n$, 
and therefore, with \eqref{join-meet-indicator}, the map 
\[
\psi:  \mathcal{H}_n \longrightarrow \mathbb{C}[S_n]
\] 
sending $x_I$ to $f_I$ for all $I \in L_n$ is well-defined. 
Indeed it gives an algebra isomorphism. See \cite[\S 2]{Hi87}. 
See also \cite[\S 3.3]{Ho-Weyl} and \cite[\S 2.3]{Ki08}.
 
\begin{proposition}\label{prop-Hibi-Sn}
The affine monoid algebra $\mathbb{C}[S_n]$ of $S_n$ is isomorphic to the 
Hibi algebra on $L_n$.
\end{proposition}

For a Young diagram $\lambda=(\lambda_1, ..., \lambda_n)$, 
let $\mathbb{C}[S_n]_{\lambda}$ denote the set of all formal linear combinations of 
GT patterns $f$ such that
\begin{equation}\label{eq-GT-top}
f(z_j^{(n)}) = \lambda_j \quad \text{\ for all $1\leq j \leq n$}.
\end{equation}
Then, the affine monoid algebra $\mathbb{C}[S_n]$ is multigraded by 
Young diagrams
\[
\mathbb{C}[S_n] = \bigoplus_{\lambda}\mathbb{C}[S_n]_{\lambda}
\]
with $\mathbb{C}[S_n]_{\lambda} \cdot \mathbb{C}[S_n]_{\mu} 
                    \subseteq \mathbb{C}[S_n]_{\lambda + \mu}$.
Similarly, $\mathcal{H}_n$ is multigraded by Young diagrams
\[
\mathcal{H}_n = \bigoplus_{\lambda}(\mathcal{H}_n)_{\lambda}
\]
and in this case,   
once monomials $\prod_{j=1}^r x_{I_j}$ are identified with multisubsets 
$\{I_1, ..., I_r\}$ of $L_n$, by Proposition \ref{prop-ssyt-gt},
the space $(\mathcal{H}_n)_{\lambda}$ consists of all formal 
linear combinations of semistandard Young tableaux on the Young diagram $\lambda$.

We remark that for each Young diagram $\lambda$, 
there is a finite dimensional irreducible representation $V^{\lambda}_n$
of the general linear group $GL_n(\mathbb{C})$.
Such a representation has  a $\mathbb{C}$-basis which can be labeled by
GT patterns satisfying \eqref{eq-GT-top}, or equivalently, 
by semistandard Young tableaux on the Young diagram $\lambda$.
Therefore, we can think of $\mathbb{C}[S_n]_{\lambda}$ and $(\mathcal{H}_n)_{\lambda}$ as 
combinatorial models of the representation space $V^{\lambda}_n$. 
We will be more precise about it in the next section.

\bigskip


\section{Flag algebra and representations of $GL_n$}\label{sec-alg-Lmn}

Let us consider the complex general linear group $GL_n=GL_n(\mathbb{C})$, that is, the group of 
complex $n \times n$ invertible matrices with matrix multiplication. We will construct algebras 
carrying polynomial representations of $GL_n$.

\subsection{Representations of $GL_n$}\label{sec-repn-gl}
Let us begin with some basic concepts of representation theory. 
For more details, we refer the reader to   \cite{GW09} especially, \S 1.5, \S 2.1, and \S 3.2. 

A \emph{representation} of a group $G$ on a vector space $V$ (over $\mathbb{C}$ in this paper)  
is a group homomorphism $\phi$ from $G$ to the group of all  automorphisms of $V$.
Then, $G$ acts on $V$ by 
\[
g\cdot v  = \phi(g) v \quad \text{for $g \in G$ and $v\in V$}.
\]
When such an action is understood, we often say $V$ is a representation of $G$. 
A representation $V$ is \emph{irrducible} 
if it is a nonzero representation that has no proper subrepresentation 
 closed under the action of $G$. 
We will focus on 
the \emph{polynomial representation} of $GL_n$, which means that 
the matrix coefficients of $\phi(g)$ of a typical element 
$g=(g_{ij}) \in GL_n$ are polynomials generated by $g_{ij}$.

We let $A_n$ be the \emph{maximal torus} of $GL_n$ consisting of all  invertible diagonal matrices. 
A \emph{character} of $A_n$ is a regular homomorphism 
\[
\psi^{\kappa}_n: A_n \longrightarrow \mathbb{C}^{\times} \quad \text{defined by \ }
 \psi_n^{\kappa}({a})= a_1^{\kappa_1} a_2^{\kappa_2} \cdots a_n^{\kappa_n}
\]
for some $\kappa=(\kappa_1, \kappa_2, ..., \kappa_n) \in \mathbb{Z}^n$. 
Here, ${a}=diag(a_{11}, ..., a_{nn}) \in A_n$ is the diagonal matrix 
with diagonal entries $a_{11}, a_{22}, ..., a_{nn}$. 
We call $\psi^{\kappa}_n$ a \emph{polynomial dominant character}, if 
$\kappa_1 \geq \kappa_2 \geq \cdots \geq \kappa_n \geq 0$.
Note that the set of all polynomial dominant characters form a semigroup
\[
\hat{A}^+_n = \left\{ \psi_n^{\kappa} : \kappa_1 \geq \kappa_2 \geq \cdots \geq \kappa_n \geq 0 \right\}
\quad \text{with \ }  \psi_n^{\alpha} \cdot  \psi_n^{\beta} =  \psi_n^{\alpha + \beta}.
\]
Every polynomial representation of $GL_n$ has a $\mathbb{C}$-basis, 
called \emph{weight basis}, consisting of vectors $v$ satisfying
\[
{a}\cdot v = \psi_n^{\kappa}({a}) v  \quad \text{for all $a \in {A_n}$}
\] 
for some $\kappa =(\kappa_1, \kappa_2, ..., \kappa_n) \in \mathbb{Z}^n$ 
such that $\kappa_i \geq 0$ for all $i$. 
Such a vector $v$ is called a \emph{weight vector} of \emph{weight} $\psi^{\kappa}_n$.

Now we let $U_n$ be the \emph{maximal unipotent  
subgroup} of $GL_n$ consisting of all upper triangular matrices with $1$'s on the diagonal. 
Let us write $V^{U_n}$  for the subspace of a polynomial representation $V$ of $GL_n$ 
consisting of all vectors invariant under the action of $U_n$.  
Since $A_n$ normalizes $U_n$, the action of $A_n$ will leave $V^{U_n}$ invariant.
Moreover, if $V$ is irreducible, then Theorem of the Highest Weight shows that 
$V^{U_n}$ is a one-dimensional subspace of $V$ spanned by a weight vector 
of weight  $ \psi_n^{\kappa} \in \hat{A}^+_n$ 
and  determines the representation $V$ up to equivalence. 
In this case, we call a vector $v \in V^{U_n}$ a \emph{highest weight vector} for $V$ 
and $\psi_n^{\kappa}$ the \emph{highest weight} of $V$.
Characters of $A_n$ occurring as the highest weights of irreducible polynomial representations 
of $GL_n$ are exactly polynomial dominant characters.

Therefore,  we can associate each irreducible polynomial representation of $GL_n$
with a polynomial dominant character $\psi_n^{\kappa}$ or equivalently 
a sequence $\kappa$ of non-increasing non-negative integers. 
We can further identify such a sequence $\kappa$ with a Young diagram as given in \S \ref{sec-yd-tab}, 
and this establishes a one-to-one correspondence between  irreducible polynomial representations of $GL_n$ 
and Young diagrams with not more than $n$ rows.
From now on, for a Young diagram $\lambda$ with $\mathbf{d}(\lambda) \leq n$, we let $V^{\lambda}_n$ 
denote the irreducible polynomial representation with highest weight $\psi_n^{\lambda}$.

\subsection{Weight basis and GT patterns}\label{sec-weight-basis}

There is a nice labeling system for weight basis elements for $V^{\lambda}_n$.
First, we recall a simple branching rule. See \cite[\S 8]{GW09}. 

\begin{lemma}\label{lem-Pieri}
For a Young diagrams $\mu=(\mu_1, ..., \mu_k)$, 
the irreducible representation $V^{\mu}_k$ of $GL_k$, 
under the restriction of  $GL_k$ down to its block diagonal subgroup $GL_{k-1} \times GL_1$, 
decomposes in a multiplicity-free fashion 
\[
V^{\mu}_k = \bigoplus_{\nu} V^{\nu}_{k-1} \otimes V^{(r)}_1
\]
where the sum is over Young diagrams $\nu=(\nu_1, ..., \nu_{k-1})$ interlacing $\mu$, i.e.
\[
\mu_1 \geq \nu_1 \geq \mu_2 \geq \nu_2 \geq \cdots \geq \nu_{k-1} \geq \mu_k
\]
and $r = \sum_{j=1}^k \mu_j  - \sum_{j=1}^{k-1} \nu_j$.
\end{lemma}

For each GT pattern $f \in S_n$, let us write 
\[
\lambda[k]=(f(z^{(k)}_1), f(z^{(k)}_2), ..., f(z^{(k)}_i)) \text{\ for $1\leq k \leq n$} \quad \text{and}
\quad \lambda=\lambda[n].
\]
Then, $f$ can encode the successive applications of  
Lemma \ref{lem-Pieri} for $k=n, n-1, ..., 2$. 
Since every irreducible polynomial representation of $GL_1$ is one-dimensional, 
we can find a vector $v_f$ in the chain of spaces
\begin{align*}
 V_n^{\lambda[n]} \supset &   \left( V_{n-1}^{\lambda[n-1]} \otimes V_1^{(\kappa_n)} \right) \supset
          \left( V_{n-2}^{\lambda[n-2]} \otimes V_1^{(\kappa_{n-1})} \otimes V_1^{(\kappa_n)} \right)  \\
& \supset \cdots 
\supset \left(V_1^{(\kappa_1)} \otimes V_1^{(\kappa_2)} \otimes \cdots \otimes V_1^{(\kappa_n)} \right)
\end{align*}
where
\[
\kappa_i = \sum_{j=1}^{i} f(z_j^{(i)}) - \sum_{j=1}^{i-1} f(z_j^{(i-1)}) 
\text{\ for $2 \leq i \leq n$} \quad \text{and} \quad \kappa_1 = f(z_1^{(1)}),
\]
and therefore $(\kappa_1)=\lambda[1]$. The vector $v_f$ is stable under the action of 
$A_n \cong GL_1 \times \cdots \times GL_1$ ($n$ times) 
with weight $ \psi^{\kappa}_n$ where $\kappa=(\kappa_1, ..., \kappa_n)$.

This gives a one-to-one correspondence between the set of GT patterns  
satisfying \eqref{eq-GT-top} and a weight basis for $V^{\lambda}_n$. See \cite{GT50, Mo06} for more details. 
With the correspondence \eqref{eq-usual-bij},
we can also label weight basis elements with semistandard Young tableaux. 
In this case, semistandard Young tableaux on a Young diagram $\lambda$ with content $\kappa$ 
corresponds to weight basis elements in $V^{\lambda}_n$ with weight $\psi^{\kappa}_n$. 
We refer the reader to \cite[\S 8.1]{GW09}.

\subsection{Flag algebra for $GL_n$}

To construct an algebra carrying irreducible representations of $GL_n$, we recall 
the \emph{$GL_n$-$GL_m$ duality} (see \cite[\S 9.2]{GW09} and \cite{Ho95}). 
Let us write $\mathbb{C}[M_{n,m}]$ for the ring of polynomials on 
the space $M_{n,m}$ of $n \times m$ complex matrices. 
We use the coordinates $x_{ab}$ to write a typical element $X \in M_{n,m}$
\[
X= \begin{bmatrix}
x_{11} &  x_{12} & \cdots & x_{1m} \\
x_{21} &  x_{22} & \cdots & x_{2m}      \\
\vdots &   \vdots & \ddots & \vdots \\
x_{n1} &  x_{n2} & \cdots & x_{nm} 
\end{bmatrix}
\]
and let the group $GL_n \times GL_m$ act on $h \in \mathbb{C}[M_{n,m}]$ by 
\[
((g_1, g_2) \cdot h )(X) = h(g_1^t X g_2)
\]
for $(g_1,g_2)\in GL_n \times GL_m$ and  $X  \in M_{n,m}$. Then,  
as a $GL_n \times GL_m$ representation, the algebra $\mathbb{C}[M_{n,m}]$ decomposes as 
\[
\mathbb{C}[M_{n,m}] \cong \bigoplus_{\lambda} V_n^{\lambda} {\otimes} V_m^{\lambda}
\]
where the summation runs over Young diagrams $\lambda$ 
with not more than $\min(n,m)$ rows.

Now we let $n \geq m$ and consider 
the subring $\mathcal{R}_{n,m}$ of $\mathbb{C}[M_{n,m}]$ 
consisting of all polynomials invariant under the action of $1 \times U_m$, 
\begin{align*}
\mathcal{R}_{n,m} 
& = \left\{ h \in \mathbb{C}[M_{n,m}]: h(Xu) = h(X) \text{\ for all $u \in U_m$}\right\}\\
& \cong \bigoplus_{\mathbf{d}(\lambda)\leq m} 
                      V^{\lambda}_n \otimes \left( V^{\lambda}_m \right)^{U_m}.
\end{align*}
We will call $\mathcal{R}_{n,m}$ the \emph{flag algebra} for $GL_n$. 
Since $V^{\lambda}_m$ is irreducible, $\left( V^{\lambda}_m \right)^{U_m}$ is 
one-dimensional. Therefore, $\mathcal{R}_{n,m}$ contains exactly one copy of every 
polynomial representation $V^{\lambda}_n$ of $GL_n$ 
for $\lambda$ with $\mathbf{d}(\lambda)\leq m$. 

Let us write $\mathcal{R}^{\lambda}_{n,m}$ for the space of weight vectors in $\mathcal{R}_{n,m}$ 
with weight $\psi_m^{\lambda}$ under the right action of $A_m$. 
Then we obtain the graded algebra structure of $\mathcal{R}_{n,m}$ 
with respect to the semigroup $\hat{A}_m^+$
\[
\mathcal{R}_{n,m} = \bigoplus_{ \psi^{\lambda}_m \in \hat{A}_m^+} 
                    \mathcal{R}^{\lambda}_{n,m}
\]
where $\mathcal{R}^{\lambda}_{n,m} \cong V^{\lambda}_n$ 
as a representation of $GL_n$ and  $\mathcal{R}^{\lambda}_{n,m} \cdot \mathcal{R}^{\lambda}_{n,m} 
\subseteq \mathcal{R}^{\lambda + \mu}_{n,m}$.

\subsection{Standard monomial basis}

Let us review a presentation of the flag algebra $\mathcal{R}_{n,m}$. 
We are particularly interested in weight bases for  
the individual homogeneous spaces $\mathcal{R}^{\lambda}_{n,m}$ of $\mathcal{R}_{n,m}$ 
under the left action of $A_n$.
Let us consider the subposet $L_{n,m}$ of $L_n$ consisting of 
all column tableaux of depth at most $m$
\[
L_{n,m}  = \left\{ I \in L_n : \text{\ the depth of $I$ is not more than $m$} \right\}. 
\]
For a column tableau $I=[i_1, i_2, ..., i_k] \in L_{n,m}$, 
we define a function $\delta_I$ on $M_{n,m}$ by the determinant of 
the submatrix of $X=(x_{ab}) \in M_{n,m}$ obtained by selecting the rows 
$i_1, i_2, ..., i_k$ and columns $1, 2, ..., k$ of $X$
\begin{equation}\label{eq-det-delta}
\delta_I (X) = \det
\begin{bmatrix}
x_{i_1 1} & x_{i_1 2} & \cdots & x_{i_1 k} \\
x_{i_2 1} & x_{i_2 2} & \cdots & x_{i_2 k} \\
\vdots     &   \vdots     & \ddots & \vdots \\
x_{i_k 1} & x_{i_2 2} & \cdots & x_{i_k k} 
\end{bmatrix}.
\end{equation}
It is easy to check that $\delta_I \in  \mathbb{C}[M_{n,m}]$ are 
invariant under the action of $1\times U_m$  and therefore the product 
$\prod_j \delta_{I_j}$ of any finite number of such elements 
belong to $\mathcal{R}_{n,m}$. 

\begin{definition}
Let $\{I_1, ..., I_r\}$ be a multisubset of $L_{n,m}$ such that 
the depth of $I_j$ is $d_j$ and $d_1 \geq d_2 \geq \cdots \geq d_r$.
\begin{enumerate}
\item 
The \emph{shape} of the product $\prod_{j=1}^r \delta_{I_j}$ is 
the Young diagram obtained by taking the transpose of $(d_1, ..., d_r)$.
\item 
The product $\prod_{j=1}^r \delta_{I_j}$ is 
called a \emph{standard monomial} if it is not divisible by $\delta_I \delta_J$ 
for any incomparable pairs $(I,J)$ in $L_{n,m}$. Therefore, its indices  
form a multichain of the poset $L_{n,m}$ and we write $\Delta_T$ for 
$\prod_{j=1}^r \delta_{I_j}$ where
\[
 T = ( I_1 \leq_{tab} I_2 \leq_{tab} \cdots \leq_{tab} I_r ).
\]
\end{enumerate}
\end{definition}

Note that every product $\Delta=\prod_{j=1}^r \delta_{I_j}$ is a weight vector 
under the right action of $A_m$. 
The weight of $\Delta$ is $\psi_m^{\lambda}$ 
(therefore $\Delta \in \mathcal{R}^{\lambda}_{n,m}$) if and only if 
the shape of  $\Delta$ is $\lambda$. 
This shows that the space $\mathcal{R}^{\lambda}_{n,m}$  is spanned by 
the products $\Delta$ whose shapes are $\lambda$.

\begin{theorem}\label{thm-smt}
For each Young diagram $\lambda$ with $\mathbf{d}(\lambda) \leq m$,
standard monomials of shape $\lambda$ form a $\mathbb{C}$-basis 
for the homogeneous component $\mathcal{R}^{\lambda}_{n,m}$ of $\mathcal{R}_{n,m}$.
\end{theorem}
\begin{proof}
Let us give a sketch of a proof. 
For more details or different proofs, we refer the reader to, 
for example, \cite{BH93, Ei80, DEP82, Fu97, GL01, MS05, Se07}.
We can begin with a determinantal identity: 
for each incomparable pair $(I, J)$ in $L_{n,m}$, we have
\[
\delta_I \delta_J = \delta_{I \vee J} \delta_{I \wedge J} +  \sum_{E,F} c_{E,F} \delta_{E} \delta_{F}
\]
where $E \leq_{tab} (I \wedge J) \leq_{tab}  (I \vee J) \leq_{tab} F$. 
By applying it to a non-standard monomial $\Delta$ in $\mathcal{R}^{\lambda}_{n,m}$ as many as possible, 
we can express $\Delta$ as a linear combination of standard monomials of shape  $\lambda$. 
Therefore, standard monomials of shape $\lambda$ span 
the space $\mathcal{R}^{\lambda}_{n,m}$.

To show that they are linearly independent, we can use 
 a monomial order on the set of all monomials $\prod_{ij} x_{ij}^{r_{ij}}$ 
in the polynomial ring $\mathbb{C}[M_{n,m}]$. Let us consider the graded lexicographic order $\geq_{glex}$ 
with respect to the following order on the variables
\[
x_{ab} > x_{cd}  \quad \text{\ if $b < d$; or $b=d$ and $a<c$.}
\]
Then, the initial monomial $in(\delta_I)$ of $\delta_I$, that is, 
the monomial appearing in the polynomial $\delta_I$ which is greater than all other monomials 
appearing in $\delta_I$ with respect to $\geq_{glex}$, is the product of 
the diagonal entries in \eqref{eq-det-delta}
\[
in(\delta_I)= x_{i_1 1} x_{i_2 2} \cdots x_{i_k k}.
\] 
From  $in(\prod_j \delta_{I_j}) = \prod_j in(\delta_{I_j})$, one can easily compute 
the initial monomials of standard monomials and show that standard monomials of shape 
$\lambda$ have distinct initial monomials with respect to $\geq_{glex}$. 
Therefore, they are linearly independent.  
\end{proof}

We note that standard monomials are stable under the left action of $A_n$, 
and therefore standard monomials of shape $\lambda$ form a weight basis for $\mathcal{R}_{n,m}^{\lambda} \cong V_n^{\lambda}$.
By identifying multichains in $L_{n,m}$ with semistandard Young tableaux, 
we obtain the following result.

\begin{corollary}
For a Young diagram $\lambda$ with $\mathbf{d}(\lambda) \leq n$, the dimension of the representation $V_n^{\lambda}$ 
is equal to the number of  semistandard Young tableaux on the Young diagram $\lambda$ with entries from $\{ 1, 2, ..., n\}$.
\end{corollary}

\subsection{Initial algebra and toric degeneration}

Now let us consider the initial algebra of the flag algebra with respect to the monomial order $\geq_{glex}$
\[
in(\mathcal{R}_{n,m}) = \left\{ in(f):  f \in \mathcal{R}_{n,m} \right\}.
\]

\begin{theorem}
There is a flat one-parameter family of algebras whose general fiber is isomorphic to $\mathcal{R}_{n,m}$
 and special fiber is isomorphic to the initial algebra $in(\mathcal{R}_{n,m})$.
\end{theorem}
\begin{proof}
From Theorem \ref{thm-smt}, every element $f\in \mathcal{R}_{n,m}$ can be uniquely expressed 
as a linear combination of standard monomials 
\[
f = c_1 \Delta_1 + c_2 \Delta_2 + \cdots + c_k \Delta_k.
\]
Since standard monomials have distinct initial monomials,  $in(f) = in(\Delta_i)$ for some $i$. 
Also, $\Delta_i$ is the product of some $\delta_{I_{j}}$ and therefore its initial monomial is the product of $in(\delta_{I_{j}})$. 
This shows that the initial algebra $in(\mathcal{R}_{n,m})$ is generated by $in(\delta_I)$ for $I \in L_{n,m}$ 
and that the set
$\left\{ \delta_I  \in \mathcal{R}_{n,m} : I \in L_{n,m} \right\}$
forms a finite SAGBI basis for the algebra $\mathcal{R}_{n,m}$. 
This guarantees that there is a flat degeneration from the flag algebra $\mathcal{R}_{n,m}$ 
to its initial algebra $in(\mathcal{R}_{n,m})$. See \cite{CHV96, GL96, KM05, MS05}.
\end{proof}

To investigate the structure of the initial algebra $in(\mathcal{R}_{n,m})$, 
we restrict the bijection from $L_n$ to $\Lambda_n$ 
given in Theorem \ref{thm:isom-birkoff} to $L_{n,m}$.
For $I \in L_{n,m}$, since the depth of $I$ is not more than $m$, 
we have $f_I (z_j^{(n)})=0$ for all $m+1 \leq j \leq n$. By the order structure of $\Gamma_n$, 
this condition forces $f_I(z_j^{(i)})=0$ for all $j \geq m+1$.
We define the smallest subposet of $\Gamma_n$ containing $Supp(f_I)$ for all $I \in L_{n,m}$
\[
\Gamma_{n,m} = \left\{  z^{(i)}_j \in \Gamma_n :   j \leq m \right\}.
\]  
See Figure \ref{fig:Hasse-GT53}. We write $S_{n,m}$ for 
the submonoid of $S_n$ consisting of all order-preserving maps in $S_n$ whose supports are 
in $\Gamma_{n,m}$ and let  $\Lambda_{n,m}= \Lambda_n \cap S_{n,m}$. 

\begin{figure}[!ht]
\centering
\begin{tikzpicture}[scale=.86]
   \node (51) at (-4,3) {$z^{(5)}_1$};
   \node (52) at (-2,3) {$z^{(5)}_2$};
   \node (53) at (0,3) {$z^{(5)}_3$};
   \node (41) at (-3,2) {$z^{(4)}_1$};
   \node (42) at (-1,2) {$z^{(4)}_2$};
   \node (43) at (1,2) {$z^{(4)}_3$};
   \node (31) at (-2,1) {$z^{(3)}_1$};
   \node (32) at (0,1) {$z^{(3)}_2$};
   \node (33) at (2,1) {$z^{(3)}_3$};  
   \node (21) at (-1,0) {$z^{(2)}_1$};
   \node (22) at (1,0) {$z^{(2)}_2$};
   \node (11) at (0,-1) {$z^{(1)}_1$};  
   \draw (51) -- (41) -- (52) -- (42) -- (53) -- (43);
   \draw (41) -- (31) -- (42) -- (32) -- (43) -- (33);  
   \draw (31) -- (21) -- (32) -- (22) -- (33);
   \draw (21)  -- (11) -- (22);  
\end{tikzpicture}
\caption{The Hasse diagram of the poset $\Gamma_{5,3}$. 
              The elements decrease along the lines from left to right.}
\label{fig:Hasse-GT53}
\end{figure}
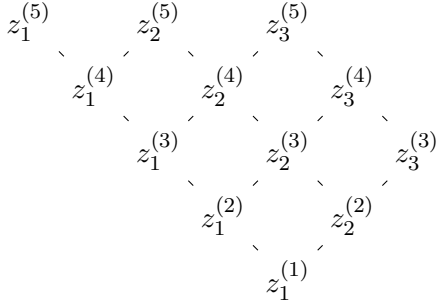

\begin{proposition}
The initial algebra $in(\mathcal{R}_{n,m})$ of the flag algebra $\mathcal{R}_{n,m}$ 
is isomorphic to the affine monoid algebra $\mathbb{C}[S_{n,m}]$ of $S_{n,m}$.
\end{proposition}
\begin{proof}
With the correspondence between $L_{n,m}$ and $\Lambda_{n,m}$, 
we define a map  $\phi: in(\mathcal{R}_{n,m}) \rightarrow \mathbb{C}[S_{n,m}]$ 
sending $in(\delta_I)$ to $f_I$ for $I \in L_{n,m}$. 
For $I = [i_1, ..., i_a]$ and $J=[j_1, ..., j_b]$ with $a \leq b \leq m$, we have
\[
in(\delta_I) in(\delta_J) = \prod_{k=1}^a (x_{i_k,k} x_{j_k,k}) \times  \prod_{k=a+1}^b x_{i_{k},k} 
                                    = in(\delta_{I \vee J}) in(\delta_{I \wedge J}).
\]
With \eqref{eq-sum-j-m22}, $\phi$ is well-defined and it can be  
extended to the initial monomials of standard monomials to give a semigroup isomorphism 
between the semigroup of the initial monomials of all $h \in {R}_{n,m}$ and $S_{n,m}$.
\end{proof}

With Proposition \ref{prop-Hibi-Sn}, this shows that the initial algebra $in(\mathcal{R}_{n,m})$ has 
the structure of the Hibi algebra on $L_{n,m}$. We also remark that its spectrum
$Spec(in(\mathcal{R}_{n,m}))$ can be understood as an affine toric variety 
associated with the lattice cone $S_{n,m}$ of GT patterns 
defined on $\Gamma_{n,m}$.

\section{More subposets of $L_{n}$ and $\Gamma_{n}$}

There are some subposets of $L_{n,m}$ and $\Gamma_{n,m}$ 
whose associated Hibi algebras  are closely related 
to important questions in invariant theory and representation theory 
of classical groups. 
In this section, we list some of them and relevant works.

\subsection{Grassmannians}

For $m \leq n$, let us consider the subposet $G_{n,m}$ of $L_{n,m}$ consisting of 
all column tableaux of depth $m$
\[
G_{n,m} = \left\{ I \in L_{n,m} : \text{\ the depth of $I$ is $m$} \right\}.
\]

Using the argument in \S \ref{sec-GT-indicator} 
we can find its associated GT poset. See  \cite[\S 3]{Ki10} and Figure \ref{fig:grss}. 
The multichains of $G_{n,m}$, the corresponding GT patterns, and the Hibi algebra attached to them 
can be used to describe the Grassmannian variety of $m$ dimensional subspaces of 
$\mathbb{C}^n$, a ring of polynomials in $\mathbb{C}[M_{n,m}]$ invariant under 
the right action of the special linear group $SL_m(\mathbb{C})$, and finite dimensional representations of 
the general linear group $GL_n(\mathbb{C})$ labeled by rectangular Young diagrams 
of depth $m$. See, for example, \cite{GL01, HP94, Stu93}. 
This poset also has an interesting connection with double tablaux or pairs 
of Young tableaux. See \cite{Ki10}.

\subsection{Symplectic groups}

For $n=2m$, let us consider the subposet of $L_{n,m}$
\[
P_n = \left\{ I \in L_{n,m}: I \geq_{tab} [1,3, 5, ..., 2m-1] \right\}.
\]

We can  find its associated GT poset using the argument in \S \ref{sec-GT-indicator}. 
See \cite{Ki08} and Figure \ref{fig:symplectic}. 
The multichains of $P_n$ and the GT patterns corresponding to them 
 can be used to label weight basis elements for 
the rational representations of the symplectic group $Sp_{n}(\mathbb{C})$. 
See, for example, \cite{Be86, DeC79, Ki08, KE83, Mo06, Pr94}.

\begin{figure}[!ht]
\centering
\begin{minipage}{.5\textwidth}
  \centering
	\begin{tikzpicture}[scale=.9]
	   \node (63) at (-1,4) {$z^{(6)}_3$};
	   \node (52) at (-2,3) {$z^{(5)}_2$};
	   \node (53) at (0,3) {$z^{(5)}_3$};
	   \node (41) at (-3,2) {$z^{(4)}_1$};
	   \node (42) at (-1,2) {$z^{(4)}_2$};
	   \node (43) at (1,2) {$z^{(4)}_3$};
	   \node (31) at (-2,1) {$z^{(3)}_1$};
	   \node (32) at (0,1) {$z^{(3)}_2$};
	   \node (33) at (2,1) {$z^{(3)}_3$};  
	   \node (21) at (-1,0) {$z^{(2)}_1$};
	   \node (22) at (1,0) {$z^{(2)}_2$};
	   \node (11) at (0,-1) {$z^{(1)}_1$};  
	   \draw (52) -- (63) -- (53);
	   \draw (41) -- (52) -- (42) -- (53) -- (43); 
	   \draw (41) -- (31) -- (42) -- (32) -- (43) -- (33);  
	   \draw (31) -- (21) -- (32) -- (22) -- (33);
	   \draw (21)  -- (11) -- (22);  
	\end{tikzpicture}
	\caption{The subposet of $\Gamma_{7,3}$ associated with $G_{7,3}$.}
	\label{fig:grss}
\end{minipage}%
\begin{minipage}{.5\textwidth}
  \centering
	\begin{tikzpicture}[scale=.86]
	   \node (61) at (-5,4) {$z^{(6)}_1$};
	   \node (62) at (-3,4) {$z^{(6)}_2$};
	   \node (63) at (-1,4) {$z^{(6)}_3$};
	   \node (51) at (-4,3) {$z^{(5)}_1$};
	   \node (52) at (-2,3) {$z^{(5)}_2$};
	   \node (53) at (0,3) {$z^{(5)}_3$};
	   \node (41) at (-3,2) {$z^{(4)}_1$};
	   \node (42) at (-1,2) {$z^{(4)}_2$};
	   \node (31) at (-2,1) {$z^{(3)}_1$};
	   \node (32) at (0,1) {$z^{(3)}_2$};
	   \node (21) at (-1,0) {$z^{(2)}_1$};
	   \node (11) at (0,-1) {$z^{(1)}_1$};
	 \draw (61) -- (51) -- (62) -- (52) -- (63) -- (53);
	 \draw (51) -- (41) -- (52) -- (42) -- (53); 
	 \draw (41) -- (31) -- (42) -- (32);  
	 \draw (31) -- (21) -- (32);
 	 \draw (21)  -- (11);  
	\end{tikzpicture}
	\caption{The subposet  of $\Gamma_{6,3}$ associated with $P_6$.}
	\label{fig:symplectic}
	\end{minipage}%
\end{figure}

\subsection{Branching rules}

For $m \leq n$ and $k < n$, let us consider the subposet $B_{n,m,k}$ of 
$L_{n,m}$ consisting of all column tableaux of the forms
\[
[1, 2, ..., p], \quad [i_1, i_2, ..., i_q], \quad [1, 2, ..., r, j_1, j_2, ..., j_s]
\]
where $1 \leq p, r \leq min(k,m)$, $1 \leq q, s \leq min(n-k,m)$, $1\leq r+s \leq m$, and 
$k+1 \leq i_c, j_d \leq n$. 
The GT poset associated with $B_{n,m,k}$ can be computed as in \S \ref{sec-GT-indicator}. 
See Figure \ref{fig:branching743}, Figure \ref{fig:branching532}, and \cite{Ki12}.

For each semistandard tableau $T$ obtained from a multichain of $B_{n,m,k}$, 
by erasing the entry $i$ in the $i$th row of $T$ for $1\leq i \leq k$ and replacing 
the entry $j$ in $T$ with $j - k$ for all $j\geq k+1$, 
we can realize $T$ as a semistandard tableau on a skew Young diagram $\lambda/\mu$ 
with content $\nu = (\nu_1, ..., \nu_{n-k})$. 
Here, the inner diagram is $\mu=(\mu_1, \mu_2,...)$ where $\mu_i$ is 
the number of all boxes in the $i$th row of $T$ containing $i$ for $1\leq i \leq k$, and 
$\nu_j$ is the number of boxes in $T$ containing $j+k$ for $1\leq j \leq n-k$.

For example, with $n=10$, $m=5$, and $k=4$, a semistandard Young tableau $T$, 
as a multichain of $B_{10,5,4}$, can be identified with a skew tableau $T'$
on a skew Young diagram $(12,10,6,4,0)/(8,5,3,0)  $ with content $(5,2,3,3,2,0)$ 
where
\[
T = \young(111111115578,2222257899,333679,5568)
\quad \leftrightarrow \quad
T' = \young(\ \ \ \ \ \ \ \ 1134,\ \ \ \ \ 13455,\ \ \ 235,1124).
\]

Then, the multichains of $B_{n,m,k}$, the corresponding GT patterns, and the Hibi algebra 
attached to  them can be used 
to describe branching rules for some pairs $(G,H)$ of classical groups, 
that is, how a representation of $G$ decomposes into 
irreducible representations of a subgroup $H$ of $G$.  See \cite{Ki12, KY12, Mo06, Pr94}.

\begin{figure}[!ht]
\centering
\begin{minipage}{.5\textwidth}
  \centering
\begin{tikzpicture}[scale=.85]
	   \node (71) at (-4,3) {$z^{(8)}_1$};
	   \node (72) at (-2,3) {$z^{(8)}_2$};
	   \node (73) at (0,3) {$z^{(8)}_3$};
	   \node (61) at (-3,2) {$z^{(7)}_1$};   
	   \node (62) at (-1,2) {$z^{(7)}_2$};
	   \node (63) at (1,2) {$z^{(7)}_3$};
	   \node (51) at (-2,1) {$z^{(6)}_1$};
	   \node (52) at (0,1) {$z^{(6)}_2$};
	   \node (53) at (2,1) {$z^{(6)}_3$};
  \node (41) at (-1,0) {$z^{(5)}_1$};
  \node (42) at (1,0) {$z^{(5)}_2$};
  \node (43) at (3,0) {$z^{(5)}_3$};
	 \draw (71) -- (61) -- (72) -- (62) -- (73) -- (63);
	 \draw (61) -- (51) -- (62) -- (52) -- (63) -- (53);  
  \draw (51) -- (41) -- (52) -- (42) -- (53) -- (43); 
\end{tikzpicture}
\caption{The subposet of $\Gamma_{8,3}$ associated with $B_{8,3,5}$.}
\label{fig:branching743}
\end{minipage}%
\begin{minipage}{.5\textwidth}
  \centering
\begin{tikzpicture}[scale=.85]
	  \node (71) at (-4,3) {$z^{(5)}_1$};
	  \node (72) at (-2,3) {$z^{(5)}_2$};
	  \node (73) at (0,3) {$z^{(5)}_3$};
	  \node (61) at (-3,2) {$z^{(4)}_1$};   
	  \node (62) at (-1,2) {$z^{(4)}_2$};
	  \node (63) at (1,2) {$z^{(4)}_3$};
	  \node (51) at (-2,1) {$z^{(3)}_1$};
	  \node (52) at (0,1) {$z^{(3)}_2$};
	  \node (53) at (2,1) {$z^{(3)}_3$};
	  \node (41) at (-1,0) {$z^{(2)}_1$};
	  \node (42) at (1,0) {$z^{(2)}_2$};
	 \draw (71) -- (61) -- (72) -- (62) -- (73) -- (63);
	 \draw (61) -- (51) -- (62) -- (52) -- (63) -- (53);  
	 \draw (51) -- (41) -- (52) -- (42) -- (53); 
\end{tikzpicture}
\caption{The subposet of $\Gamma_{5,3}$ associated with $B_{5,3,2}$.}
\label{fig:branching532}
	\end{minipage}%
\end{figure}

\subsection{Tensor product of representations}

The tensor product decomposition problem 
to determine how tensor products of group representations 
decomposes is an important problem in representation theory 
with many applications.
Recently, Howe and his collaborators have shown that 
answers to many of these questions can be given nicely in terms of 
the Hibi algebras associated with some subposets of $\Gamma_{n,m}$ 
and their variations. We refer the interested reader to 
\cite{Ho13, HL12, HKL1, HKL2, Ki18, KL13, KYo17, Wa14}.

\bigskip


\subsection*{Acknowledgement}

Parts of this article were presented at 
The Prospects for Commutative Algebra, Osaka, Japan, July 2017.
We express our sincere gratitude 
to the organizers for the wonderful and stimulating conference.

\bigskip


\bigskip


\begin{thebibliography}{99}
\bibitem{Be86} A. Berele, 
Construction of $\mathrm{Sp}$-modules by tableaux. 
Linear and Multilinear Algebra 19 (1986), no. 4, 299--307.

\bibitem{BH93} W. Bruns and J. Herzog, 
Cohen-Macaulay rings.
Cambridge Studies in Advanced Mathematics, 39. 
Cambridge University Press, Cambridge, 1993.

\bibitem{CHV96} A. Conca, J. Herzog, and G. Valla, 
Sagbi bases with applications to blow-up algebras. 
J. Reine Angew. Math. 474 (1996), 113--138.

\bibitem{DeC79} C. De Concini, 
Symplectic standard tableaux. 
Adv. in Math. 34 (1979), no. 1, 1--27.

\bibitem{DEP82} C. De Concini, D. Eisenbud, and C. Procesi, 
Hodge algebras. 
Ast\'{e}risque, 91. Soci\'{e}t\'{e} Math\'{e}matique de France, Paris, 1982. 87 pp.

\bibitem{Ei80} D. Eisenbud, 
Introduction to algebras with straightening laws. 
Ring theory and algebra, III (Proc. Third Conf., Univ. Oklahoma, Norman, Okla., 1979), pp. 243--268, 
Lecture Notes in Pure and Appl. Math., 55, Dekker, New York, 1980.

\bibitem{Fu97} W. Fulton, 
Young tableaux. With applications to representation theory and geometry. 
London Mathematical Society Student Texts, 35. Cambridge University Press, Cambridge, 1997.
 
\bibitem{GT50} I. M. Gelfand and M. L. Tsetlin, 
Finite-dimensionalrepresentations of the group of unimodular matrices. 
{Doklady Akad. Nauk SSSR} (N.S.) 71, (1950), 825--828. 
English translation in Izrail M. Gelfand, Collected papers. Vol. II. Springer-Verlag, Berlin, 1988.

\bibitem{GL96}
N. Gonciulea and V. Lakshmibai, 
Degenerations of flag and Schubert varieties to toric varieties. 
Transform. Groups 1 (1996), no. 3, 215--248. 

\bibitem{GL01} N. Gonciulea and V. Lakshmibai, 
\emph{Flag varieties}, Hermann {\'E}diteurs des Sciences et des Arts, 2001.

\bibitem{GW09} R. Goodman and N. R. Wallach, 
Symmetry, representations, and invariants. 
Graduate Texts in Mathematics, 255. Springer, Dordrecht, 2009.

\bibitem{Hi87} T. Hibi,  
Distributive lattices, affine semigroup rings and algebras with straightening laws. 
Commutative algebra and combinatorics (Kyoto, 1985), 93--109, 
Adv. Stud. Pure Math., 11, North-Holland, Amsterdam, 1987.

\bibitem{Hod43} W. V. D. Hodge, 
Some enumerative results in the theory of forms. 
Proc. Cambridge Philos. Soc. 39, (1943). 22--30. 

\bibitem{HP94} W. V. D. Hodge and D. Pedoe, 
Methods of algebraic geometry. Vol. II. 
Reprint of the 1952 original. Cambridge Mathematical Library. 
Cambridge University Press, Cambridge, 1994.

\bibitem{Ho95} R. Howe, 
Perspectives on invariant theory: Schur duality, multiplicity-free actions and beyond. 
The Schur lectures (1992) (Tel Aviv), 1--182, 
Israel Math. Conf. Proc., 8, Bar-Ilan Univ., Ramat Gan, 1995.

\bibitem{Ho-Weyl} R. Howe, 
Weyl Chambers and standard monomial theory for poset lattice cones. 
Q. J. Pure Appl. Math. 1 (2005), no. 1, 227--239. 

\bibitem{Ho13} R. Howe, 
Pieri algebras and Hibi algebras in representation theory. 
\emph{Symmetry: representation theory and its applications}, 353--384, 
Progr. Math., 257, Birkh{\"a}user/Springer, New York, 2014. 

\bibitem{HL12} R. Howe and S. T. Lee, 
Why should the Littlewood-Richardson rule be true? 
Bull. Amer. Math. Soc. (N.S.) 49 (2012), no. 2, 187--236.

\bibitem{HKL1} R. Howe, S. Kim, and S. T. Lee, 
Double Pieri algebras and iterated Pieri algebras for the classical groups. 
Amer. J. Math. 139 (2017), no. 2, 347--401.

\bibitem{HKL2} R. Howe, S. Kim, and S. T. Lee, 
Standard monomial theory for harmonics in classical invariant theory.
\emph{Representation Theory, Number Theory and Invariant Theory}, 265-302.
Progr. Math., 323, Birkh{\"a}user/Springer, New York, 2017.

\bibitem{Ki08} S. Kim, 
Standard monomial theory for flag algebras of ${GL}(n)$ and ${Sp}(2n)$. 
J. Algebra 320 (2008), no. 2, 534--568.

\bibitem{Ki10} S. Kim, 
The nullcone in the multi-vector representation of the symplectic group and related combinatorics. 
J. Combin. Theory Ser. A 117 (2010), no. 8, 1231--1247.

\bibitem{Ki12} S. Kim, 
Distributive lattices, affine semigroups, and branching rules of the classical groups. 
J. Combin. Theory Ser. A  119 (2012), 1132--1157.

\bibitem{Ki18} S. Kim, 
A presentation of the double Pieri algebra. 
J. Pure Appl. Algebra 222 (2018), no. 2, 368--381.

\bibitem{KL13} S. Kim and S. T. Lee, 
Pieri algebras for the orthogonal and symplectic groups. 
Israel J. Math. 195 (2013), no. 1, 215--245. 

\bibitem{KY12} S. Kim and O. Yacobi, 
A basis for the symplectic group branching algebra,  
J. Algebraic Combin. 35 (2012), no. 2, 269--290.

\bibitem{KYo17} S. Kim and S. Yoo, 
Pieri and Littlewood-Richardson rules for two rows and cluster algebra structure. 
J. Algebraic Combin. 45 (2017), no. 3, 887--909.

\bibitem{KE83} R. C. King and N. G. I. El-Sharkaway, 
Standard Young tableaux and weight multiplicities of the classical Lie groups. 
J. Phys. A 16 (1983), no. 14, 3153--3177. 

\bibitem{KM05}
M. Kogan and E. Miller, 
Toric degeneration of Schubert varieties and Gelfand-Tsetlin polytopes. 
Adv. Math. 193 (2005), no. 1, 1--17. 

\bibitem{MS05} E. Miller and B. Sturmfels, 
Combinatorial commutative algebra. 
Graduate Texts in Mathematics, 227. Springer-Verlag, New York, 2005

\bibitem{Mo06} A. I. Molev, 
Gelfand-Tsetlin bases for classical Lie algebras. 
Handbook of algebra. Vol. 4, 109--170, Handb. Algebr., 4, Elsevier/North-Holland, Amsterdam, 2006.

\bibitem{Pr94} R. A. Proctor, 
Young tableaux, Gelfand patterns, and branching rules for classical groups. 
J. Algebra 164 (1994), no. 2, 299--360.

\bibitem{Se07} C. S. Seshadri, 
Introduction to the theory of standard monomials. 
Texts and Readings in Mathematics, 46. Hindustan Book Agency, New Delhi, 2007.

\bibitem{St12} R. P. Stanley, 
Enumerative combinatorics. Vol. 1. Second edition. 
Cambridge Studies in Advanced Mathematics, 49. Cambridge University Press, Cambridge, 2012

\bibitem{St99} R. P. Stanley, 
Enumerative combinatorics. Vol. 2. 
Cambridge Studies in Advanced Mathematics, 62. Cambridge University Press, Cambridge, 1999.

\bibitem{Stu93} B. Sturmfels, 
Algorithms in invariant theory. 
Texts and Monographs in Symbolic Computation. Springer-Verlag, Vienna, 1993.

\bibitem{Wa14} Y. Wang, 
Sign Hibi cones and the anti-row iterated Pieri algebras for the general linear groups. 
J. Algebra 410 (2014), 355--392.
\end{thebibliography}
\end{document}